\documentclass{amsart}
\usepackage{amssymb,latexsym,eucal,mathrsfs}

\usepackage[curve,v2]{xy}

\newcommand{\tensor}{\otimes}

\newcommand{\wt}[1]{\widetilde{#1}}

\newcommand{\wts}{\widetilde{\operatorname{Sec}X}}

\renewcommand{\H}{{\mathcal H}}

\newcommand{\D}{{\mathscr D}}

\renewcommand{\O}{{\mathcal O}}
\renewcommand{\P}{{\mathbb{P}}}
\newcommand{\C}{{\mathbb{C}}}

\newcommand{\G}{{\mathbb{G}}}

\renewenvironment{proof}{\par \medskip \noindent
{\sc Proof:}}{\nopagebreak \hfill $\Box$ \par \medskip}

\begingroup
\newtheorem{thm}{Theorem}[section]   
\newtheorem{cor}[thm]{Corollary}     
\newtheorem{lemma}[thm]{Lemma}         
\newtheorem{prop}[thm]{Proposition}  
\newtheorem{defn}[thm]{Definition}   
   
\endgroup

\newenvironment{rem}[2]{\refstepcounter{thm} \label{#2} 
\par \medskip \noindent {\bf #1 \thethm }}{\par \medskip}

\begin{document}



\pagenumbering{arabic}

\title[]{Singularities of the Secant Variety}

\author[]{Peter Vermeire}

\address{Pearce 214, Central Michigan
University, Mount Pleasant MI 48859}

\email{verme1pj@cmich.edu}

\date{\today}

\begin{abstract}
We give positivity conditions on the embedding of a smooth variety which guarantee the normality of the secant variety, generalizing earlier results of the author and others.  We also give classes of secant varieties satisfying the Hodge conjecture as well as a result on the singular locus of degenerate secant varieties.
\end{abstract}


\maketitle

\section{Introduction}
We work throughout over the field $\C$ of complex numbers.
Our starting point is the following corollary of a result \cite[1.1]{bertram1} of A. Bertram for complete linear systems and \cite{coppens} of M. Coppens in general:
\begin{thm}\label{start}\cite[3.2]{vermeireideals}
If $X\subset\P^n$ is a $2k$-very ample embedding of a smooth curve, then $\operatorname{Sec}^{k-1}X$ is normal and $\operatorname{Sec}^{k-1}X\setminus \operatorname{Sec}^{k-2}X$ is smooth.
\end{thm}

For $g=1$ this has been reproved in \cite[8.15-8.16]{vbh} and in \cite{fisher} where it is further shown that the secant varieties are arithmetically Gorenstein.
Theorem~\ref{start} was extended in \cite[3.10]{vermeireflip1} to:
\begin{prop}\label{alliswell}
Let $X\subset\P^n$ be a smooth projective variety satisfying $K_2$.  If the embedding is $4$-very ample, then $\operatorname{Sec}X$ is normal and $\operatorname{Sec}X$ is smooth off $X$.
\end{prop}

Recall that a variety $X\subset\P^n$ satisfies $K_2$ if it is scheme-theoretically cut out by quadrics and the Koszul relations among the quadrics are generated by linear syzygies.  Condition $K_2$ is implied by the much better known condition $N_2$, where the embedding is projectively normal, the ideal is generated by quadrics, and all relations are generated by linear syzygies.  The connection between Theorem~\ref{start} and Proposition~\ref{alliswell} is that on a smooth curve, a line bundle of degree $2g+3$ induces a $4$-very ample embedding that satisfies $N_2$ \cite{mgreen}.  It is typically easier to verify $4$-very ampleness than it is to verify $K_2$.  Conversely, many standard embeddings (e.g. Veronese, Segre, Pl\"ucker) satisfy $K_2$, but are not $4$-very ample.  
For example, in \cite{landwey}, results analogous to Proposition~\ref{alliswell} are shown for some products of projective spaces, which satisfy $K_2$ but not the $4$-very ample hypothesis:
\begin{thm}\cite[1.1,1.2]{landwey}
If $X$ is the Segre embedding of $\P^1\times\P^a\times\P^b$, then $\operatorname{Sec}^rX$ is normal with rational singularities for all $r\geq1$.  If $X$ is the Segre embedding of $\P^a\times\P^b\times\P^c\times\P^d$, then $\operatorname{Sec}X$ is arithmetically Cohen-Macaulay.
\nopagebreak \hfill $\Box$ \par \medskip
\end{thm}

Here we separate the two conditions in Proposition~\ref{alliswell} and give a result (Theorem~\ref{newmain}) absorbing the first stage of the results above.

We make use of the following elementary, and probably well-known, result:
\begin{lemma}\label{normality}
Let $f:X\rightarrow Y$ be a proper surjective
morphism of irreducible varieties over an algebraically closed field $k$
with reduced, connected fibers. If $X$ is normal then
$Y$ is normal.
\end{lemma}

\begin{proof}

Because $X$ is normal, $f$ factors uniquely through the
normalization $\wt{Y}$ of $Y$:
\begin{center}
{\begin{minipage}{1.5in}
\diagram
  X \drto_f \rto^{\wt{f}} & \wt{Y}
\dto^g \\
    & Y
\enddiagram
\end{minipage}}
\end{center}
where $g$ is a finite, hence proper, morphism.  The properness of $f$
and $g$ implies that $\wt{f}$ is proper (\cite[II.4.8.e]{hart}),
hence its image is closed in $\wt{Y}$.  Surjectivity of $f$
implies that $dim~Im(\wt{f}) = dim~Y =
dim~\wt{Y}$, and so $\wt{f}$ is surjective. 
Because $f$ has reduced, connected fibers, $g$ does also, and so as $g$
is finite the fibers of $g$ are reduced points.
\end{proof}

\section{Results}

\subsection{Construction of the Secant Varieties}

For $X\subset\P^n$ a projective variety, 
we recall a construction from \cite{vermeireflip1} describing a vector bundle $\mathscr{E}^k_{\O_X(1)}$ on $\H^{k}X=\operatorname{Hilb}^{k}(X)$ and a morphism $$\P_{\H^kX}\left(\mathscr{E}^k_{\O_X(1)}\right)\rightarrow\P^n$$ whose image is the full $(k-1)$st secant variety to $X$ in $\P^n$.

Let $L$ be an invertible sheaf on $X$, and denote by $\D_k$ the universal
subscheme of $X\times \H^{k}X$.
Let $\pi:X\times \H^{k}X\rightarrow
X$ and $\pi_{2}:X\times \H^{k}X\rightarrow \H^{k}X$ be the
projections.  Form the
invertible sheaf $\O_{\D_k}\tensor \pi^*L$ on $\D_k
\subset X\times \H^{k}X$. Now $\pi_{2}|_{\D_k}:\D_k
\rightarrow \H^{k}X$ is flat of degree $k$, hence
$\mathscr{E}^k_L=(\pi_{2})_*(\O_{\D_k}\tensor \pi^*L)$ is a
locally free sheaf of rank $k$ on $\H^{k}X$.  We define the {\boldmath\bf full $(k-1)$st 
secant bundle of $X$ with respect to $L$} to be the $\P^{k-1}$-bundle
$B^{k-1}(L)=\P_{\H^{k}X}(\mathscr{E}^k_L)$. 

Suppose $V\subseteq \Gamma(X,L)$ is a $k$-very ample linear system.
To define the desired map to $\P(V)$, push the natural restriction  
$\pi^*L\rightarrow \O_{\D_k}\tensor \pi^*L$
down to $\H^{k}X$ giving an evaluation map
$H^0(X,L)\tensor \O_{\H^{k}X}\rightarrow \mathscr{E}^k_L$
which in turn for any linear system $V\subseteq H^0(X,L)$ restricts to
$V\tensor \O_{\H^{k}X}\rightarrow \mathscr{E}^k_L$,
A fiber of $\mathscr{E}^k_L$ over a point $Z\in \H^{k}X$ is $H^0(X,L\tensor \O_Z)$, so since $V$ is $k$-very ample this map is surjective 
and we obtain a morphism:
$$\beta_{k-1}:B^{k-1}(L)\rightarrow \P(V)\times \H^{k}X\rightarrow \P(V)$$
The image of this morphism
is the \textbf{full $(k-1)$st secant variety to $X$ in $\P(V)$}.  Note that as long as $V$ is $r$-very ample, $\beta_{r-2}$ restricts to $\beta_{r-2}:\D_{r-1}\rightarrow X$.

Note that as the Hilbert scheme may have several components (e.g. \cite{iarrobino}), this not just the closure of the locus spanned by distinct points.  The usual $(k-1)$st secant variety $\operatorname{Sec}^{k-1}X$ is obtained by restricting to the component of the Hilbert scheme containing the points corresponding to reduced subschemes.  For the rest of this paper, when we refer to $\beta_k$, we mean the restriction of the $\beta_k$ constructed above to the bundle on this component of the Hilbert scheme. 

\begin{rem}{Remark}{hodge}
Combining this construction with the motivic invariant constructed in \cite{arapura}, we obtain the following collection of secant varieties which satisfy the (generalized) Hodge conjecture.  All of the results are obtained via computation of the invariant on Hilbert schemes.
\begin{enumerate}
\item Let $X\subset\P^n$ be a smooth surface with $p_g=0$.  Then $\operatorname{Sec}X$ satisfies the generalized Hodge conjecture.
\item Let $X\subset\P^n$ be a smooth curve of genus at most $2$ embedded by a line bundle of degree at least $2g+k$.  Then $\operatorname{Sec}^kX$ satisfies the generalized Hodge conjecture.
\item Let $X\subset\P^n$ be a smooth curve of genus $3$ embedded by a line bundle of degree at least $2g+k$.  Then $\operatorname{Sec}^kX$ satisfies the Hodge Conjecture.
\item Let $X\subset\P^n$ be a smooth curve.  Then $\operatorname{Sec}X$ satisfies the generalized Hodge Conjecture.  If the embedding is $3$-very ample then $\operatorname{Sec}^2X$ satisfies the Hodge Conjecture.
\nopagebreak \hfill $\Box$ \par \medskip
\end{enumerate}
\end{rem}

As a consequence of the above construction, we have:
\begin{thm}\label{newmain}
Let $X\subset\P^n$ be a smooth projective variety.  If the embedding is $4$-very ample then $\operatorname{Sec}X$ is normal and $\operatorname{Sec}X$ is smooth off $X$.  If the embedding satisfies $K_2$ then $\operatorname{Sec}X$ is normal. 
\end{thm}

We discuss smoothness of the secant variety for embeddings satisfying $K_2$ but which are not $4$-very ample below (Theorem~\ref{smoothdeficient}).

\begin{proof}(of Theorem~\ref{newmain})
Suppose the embedding is $4$-very ample.  Then this implies that the restriction $\beta_1:B^1(L)\setminus\D\rightarrow\operatorname{Sec}X\setminus X$ is an isomorphism since, by the above construction, $\beta_1$ separates secant and tangent lines off of $X$.  Therefore $\operatorname{Sec}X$ is smooth off $X$.  Because, in particular, the embedding is $3$-very ample, we have the restriction $\beta_{1}:\D_{2}\rightarrow X$.  Since $\D_2=\operatorname{Bl}_{\Delta}\left(X\times X\right)$ (e.g. \cite{gott}), we see that if $x\in X$ then $\beta_{1}^{-1}(x)=\operatorname{Bl}_xX$.  Thus $\operatorname{Sec}X$ is normal by Lemma~\ref{normality}.

Suppose the embedding satisfies $K_2$, and let $x\in\operatorname{Sec}X\setminus X$.  By \cite{vermeireflip1}, $x$ lies in a unique linear $\P^k$ spanned by a quadric hypersurface and the scheme-theoretic fiber is $\beta_1^{-1}(x)=\P^{k-1}$.  If $x\in X$, then $\beta_{1}^{-1}(x)=\operatorname{Bl}_LX$ where $L$ is the largest linear space in $X$ containing $x$.  Thus $\operatorname{Sec}X$ is normal by Lemma~\ref{normality}.
\end{proof}





An obvious problem in attempting to extend Theorem~\ref{start} to varieties of larger dimension is the fact that $\operatorname{Hilb}^rX$ is singular whenever both $r>3$ and $\operatorname{dim}X>2$.  However, we can give a quick proof of a fact found in \cite[16.17]{harris}:






\begin{prop}
Let $X\subset\P^n$ be a $2k$-very ample embedding of a smooth projective variety.  Then $\operatorname{Sec}^{k-1}X$ is smooth away from $\operatorname{Sec}^{k-2}X$.
\end{prop}

\begin{proof}
By the $2k$-very ample hypothesis, $\beta_{k-1}$ separates $k$-secant $\P^{k-1}$s away from $\operatorname{Sec}^{k-2}X$, hence is an isomorphism away from $\beta_{k-1}^{-1}\left(\operatorname{Sec}^{k-2}X\right)$.   Therefore $\operatorname{Sec}^{k-1}X$ is smooth away from $\operatorname{Sec}^{k-2}X$.
\end{proof}

\subsection{Deficient Secant Varieties}

For varieties that satisfy $K_2$, the deficiency of the secant variety is nicely tied to the geometry of the embedding:

\begin{prop}\cite[3.12]{vermeireflip1}
Suppose $X\subset\P^n$ is a smooth, irreducible variety satisfying condition $K_2$.  Then $\operatorname{dim}(\operatorname{Sec}X)=2\operatorname{dim}(X)+1-r$, where the generic pair of points of $X$ lies on a quadric hypersurface of (maximal) dimension $r$.
\nopagebreak \hfill $\Box$ \par \medskip
\end{prop}

We are interested in the following class of varieties
\begin{defn}\label{covered}
A projective variety $X\subset\P^n$ is \textbf{covered by r-conics }if every length two subscheme of $X$ lies on a quadric hypersurface of dimension $r$, and no length two subscheme of $X$ lies on a quadric hypersurface of dimension $r+1$.
\end{defn}

\begin{rem}{Example}{r-conics}
Non-degenerate curves $C\subset\P^n$, $n>2$, are covered by 0-conics.  $v_2(\P^n)\subset\P^N$, $n\geq 1$, is covered by 1-conics, while $v_k(\P^n)\subset\P^N$, $n\geq1$, $k>2$, is covered by 0-conics.

Other simple examples are given by the Severi varieties \cite{lazvan}: $v_2(\P^2)\subset\P^5$ is covered by 1-conics (as above); $\sigma_{1,1}(\P^2\times\P^2)\subset\P^8$ is covered by 2-conics; $\mathbb{G}(1,5)\subset\P^{14}$ is covered by 4-conics; and $\mathbb{O}\mathbb{P}^2\subset\P^{26}$ is covered by 8-conics.
\nopagebreak \hfill $\Box$ \par \medskip
\end{rem}

\begin{rem}{Remark}{qel}
Note that this is a strengthening of the condition that $X$ is a quadratic entry locus (QEL) variety \cite{ionescurusso},\cite{ionescurusso2},\cite{russo},\cite{fu}, where the condition in Definition~\ref{covered} is weakened to hold for general points and without restrictions on the dimension of the quadrics.
\nopagebreak \hfill $\Box$ \par \medskip
\end{rem}

Given a linearly normal variety $X\subset\P^n$ satisfying $K_2$ and covered by r-conics, we may construct the secant variety $\operatorname{Sec}X\subset\P^n$ as follows:  Let $\mathcal{H}\subset\operatorname{Hilb}(X)$ be the component of the Hilbert scheme  parameterizing quadric hypersurfaces of dimension $r$ and let $D\subset X\times\mathcal{H}$ be the universal subscheme.  Then $\mathbb{P}_{\mathcal{H}}\left(\left(\pi_2\right)_*\left(\O_D\tensor\pi_1^*\O_X(1)\right)\right)$ is a secant bundle on $\mathcal{H}$ (which is in some sense compatible with the secant bundle constructed earlier).  Pushing the restriction $\pi_1^*\O_X(1)\rightarrow \O_D\tensor\pi_1^*\O_X(1)$ to $\mathcal{H}$ gives the surjection $\Gamma(X,\O_X(1))\tensor\O_{\mathcal{H}}\rightarrow \left(\pi_2\right)_*\left(\O_D\tensor\pi_1^*\O_X(1)\right)$.  This induces the morphism 
$$\mathbb{P}_{\mathcal{H}}\left(\left(\pi_2\right)_*\left(\O_D\tensor\pi_1^*\O_X(1)\right)\right)\rightarrow \mathbb{P}\Gamma(X,\O_X(1))\times\mathcal{H}\rightarrow \mathbb{P}\Gamma(X,\O_X(1))$$ whose image is the secant variety.

If $X\subset\P^n$ is a smooth variety satisfying $K_2$, then the morphism $\varphi:\operatorname{Bl}_X\P^n=\wt{X}\rightarrow\P^s$ determined by $\O_{\wt{X}}(2H-E)$ is an embedding off $\wts$ \cite[2.11]{vermeireflip1}.  Further, we have:

\begin{prop}\label{secantstructure}
Suppose $X\subset\P^n$ is a smooth, irreducible, non-degenerate variety satisfying condition $K_2$.   If $X$ is covered by $r$-conics then the restriction $\varphi:\wts\rightarrow \P^s$ is a $\P^{r+1}$-bundle over the component $\mathcal{H}$ of the Hilbert scheme of $X$ parameterizing conic hypersurfaces of dimension $r$.  Further, the exceptional locus of $\pi:\wts\rightarrow\operatorname{Sec}X$ is the universal subscheme of $X\times\mathcal{H}$.
\end{prop}

\begin{proof}
This follows exactly as in \S 3 of \cite{vermeireflip1} using the description of the secant variety above.  In particular,  $\wts=\P_{\mathcal{H}}\left(\varphi_*\O_{\wts}(H)\right)$.
\end{proof}

\begin{rem}{Example}{quadveron}
Some interesting special cases are:
\begin{enumerate}
\item In case $r=0$, $\wts$ is a $\P^1$-bundle over $\operatorname{Hilb}^2X$ \cite[3.9]{vermeireflip1}.  
\item When $X=v_2(\P^n)$, $\wts$ is a $\P^2$-bundle over $\mathbb{G}(1,n)$ \cite[4.14]{vermeireflip1} (note that when $n=2$, this is $\left(\P^2\right)^*$).  
\item When $X=\sigma_{1,1}\left(\P^2\times\P^2\right)\subset\P^8$, $\wts$ is a $\P^3$-bundle over $\left(\P^2\right)^*\times\left(\P^2\right)^*$.
\item When $X=\G(1,5)\subset\P^{14}$, $\wts$ is a $\P^5$-bundle over $\G(3,5)$.
\end{enumerate}
\nopagebreak \hfill $\Box$ \par \medskip
\end{rem}

In general, we do not know how to determine when $\mathcal{H}$ is smooth, except that it is easy to check that if $Q\subset X$ a quadric hypersurface, then $H^1(Q,T^1_X\tensor\O_Q)=H^1(Q,N_{Q/X})$.

By our hypothesis that $X\subset\P^n$ is covered by $r$-conics, the Hilbert functor of dimension $r$ conics on $X$ is a subfunctor of the projective Grassmannian functor on $\P^n$ represented by $\mathbb{G}(r+1,n)$.  As the universal subscheme of $\P^n\times\mathbb{G}(r+1,n)$ is the tautological $\P^{r+1}$-bundle over $\mathbb{G}(r+1,n)$, this says that $\mathcal{H}\hookrightarrow\mathbb{G}(r+1,n)$ and that the universal subscheme of $X\times\mathcal{H}$ is a conic bundle over $\mathcal{H}$ of dimension $r$.  Again, for $r=0$ this is \cite[\S 3]{vermeireflip1}.


\begin{cor}\label{smoothdeficient}
Suppose $X\subset\P^n$ is a smooth, irreducible, non-degenerate variety satisfying condition $K_2$ that is covered by $r$-conics.  If $\mathcal{H}$ is smooth, then $\wts$ is smooth, hence $\operatorname{Sec}X$ is smooth off $X$.  
\nopagebreak \hfill $\Box$ \par \medskip
\end{cor}




\textbf{Acknowledgements:} I was led revisit these questions after an exchange with J.M. Landsberg.  I would also like to thank for D. Arapura for communication regarding the Remark~\ref{hodge}.

\end{document}